\documentclass[11pt]{article}
\usepackage{amssymb,amsmath,amsfonts}
\usepackage{tikz}
\usepackage{graphics}
\usepackage{graphicx}

\DeclareMathOperator{\hol}{Hol}

\def\eqref#1{$(\ref{#1})$}

\newenvironment{proof}{\noindent {\em {Proof}}.}{$\square$
\medskip}
\newtheorem{theorem}{Theorem}[section]
\newtheorem{corollary}[theorem]{Corollary}
\newtheorem{lemma}[theorem]{Lemma}
\newtheorem{remark}[theorem]{Remark}
\newtheorem{proposition}[theorem]{Proposition}

\newtheorem{example}[theorem]{Example}
\newtheorem{question}[theorem]{Question}
\newtheorem{problem}[theorem]{Problem}

\newcommand{\pdisk}{\mathbb{D}^n}
\newcommand{\disk}{\mathbb{D}}
\newcommand{\mua}{\mu_\alpha}

\newcommand{\bt}[1]{\begin{theorem}\label{#1}}
\newcommand{\bc}[1]{\begin{corollary}\label{#1}}
\newcommand{\bl}[1]{\begin{lemma}\label{#1}}
\newcommand{\bp}[1]{\begin{proposition}\label{#1}}
\newcommand{\be}[1]{\begin{example}\rm\label{#1}}
\newcommand{\bq}[1]{\begin{question}\rm\label{#1}}
\newcommand{\bprob}[1]{\begin{problem}\rm\label{#1}}
\newcommand{\beq}[1]{\begin{eqnarray}\label{#1}}
\newcommand{\br}[1]{\begin{remark}\rm\label{#1}}

\newcommand{\el}{\end{lemma}}
\newcommand{\ep}{\end{proposition}}
\newcommand{\ee}{\end{example}}
\newcommand{\eq}{\end{question}}
\newcommand{\eprob}{\end{problem}}
\newcommand{\eeq}{\end{eqnarray}}
\newcommand{\ed}{\end{definition}}
\newcommand{\et}{\end{theorem}}
\newcommand{\ec}{\end{corollary}}
\newcommand{\er}{\end{remark}}

\title{Approximation in weighted holomorphic Besov spaces in $\mathbb{C}^n$}
\author{Ali Abkar
\vspace{0.5in}\\
Department of Pure Mathematics, Faculty of Science,\\
Imam Khomeini International University, Qazvin 34149, Iran\\
\\
\\
\texttt{Email:~abkar@sci.ikiu.ac.ir}\\
 }
\date{}
\begin{document}
\maketitle \textbf{Abstract.} We study certain weighted Bergman and weighted Besov spaces of holomorphic functions in the polydisk and in the unit ball. We seek Mergelyan-type conditions on the non-radial weight function to guarantee that the dilations of a given function tend to the same function in norm; in particular, we seek conditions on the non-radial weights to ensure that the
analytic polynomials are dense in the space.\\

\noindent\textbf{Keywords}: {Bergman space, analytic Besov space, dilation, non-radial weight, angular weight}\\

\noindent \textbf{MSC2020}: 46E15, 46E20, 30H25, 32A36\\
\textbf{}{}

\section{Introduction}

In the study of Banach spaces of analytic functions on sub-domains of $\mathbb{C}^n$, an important question is to know whether the polynomials are dense in the space. In the complex plane, the
first result on the approximation by polynomials in the space of $p$-th power area integrable functions on simply connected domains was obtained by Torsten Carleman in his 1923 paper \cite{car}. Indeed, Carleman approximated functions in the Bergman space of the domain, but at the time the phrase "Bergman space" was not yet coined.
Carleman's result was then extended to more general regions by O.J. Farrell (see \cite{far1}, \cite{far2}).
Farrell proved that for each $p$-th power area integrable function $f$, there is a sequence of polynomials $p_n$ such that $p_n\to f$ in norm.
We should remark that the Taylor polynomials of a given function do not necessarily approximate the function in norm (see \cite{pzz}, and \cite{zhu}). Instead, one tries to show that the dilations $f_r(z)=f(rz)$ for $0\le r<1$, converge to $f$ in norm. It is easily understood that the dilations $f_r$ that are defined on bigger domains are more well-behaved and tractable than the function $f$ itself. 
\par
In this paper we are concerned with approximation in weighted Bergman and weighted Besov spaces in several complex variables.
We mean by the polydisk the subset
$$\mathbb{D}^n=\{z=(z_1,...,z_n)\in\mathbb{C}^n: |z_j|<1,\,\, 1\le j\le n\}$$
of the $n$-dimensional complex space, and by the unit ball the
set
$$\mathbb{B}_n=\{z=(z_1,...,z_n)\in\mathbb{C}^n: |z|^2=|z_1|^2+\cdots+|z_n|^2<1\}.$$
Let $\hol(\disk^n)$ denote the space of holomorphic functions on $\disk^n$.
For $0<p<\infty$,
the Bergman space on the polydisk $\disk^n$ is defined as
$$A^p_\alpha(\disk^n)=\hol(\disk^n)\cap L^p(\disk ^n, dV_\alpha)$$
where $dV_\alpha(z) =dA_\alpha(z_1)\cdots dA_\alpha (z_n)$, and
$$dA_\alpha(z_k)=(1-|z_k|^2)^\alpha dx_k\,dy_k.$$
In general, $\alpha$ is bigger than $-1$, but in this paper we shall assume that $\alpha\ge 0$. If $\alpha=0$, we write
$dV(z)=dA(z_1)\cdots dA(z_n)$.
\par
Let $w$ be a positive function on $\disk^n$. For $0<p<\infty$, the weighted Bergman space $A^p_w(\disk^n,dV_\alpha)$ consists of holomorphic functions $f$ on the polydisk such that the integral
$$\Vert f\Vert^p_{A^p_w(\disk^n,dV_\alpha)}=\int_{\disk^n}|f(z)|^pw(z)dV_\alpha(z)$$
is finite. In a similar fashion, 
let $\hol(\mathbb{B}_n)$ denote the space of holomorphic functions on the unit ball.
The Bergman space on the unit ball is defined as
$$A^p(\mathbb{B}_n)=\hol(\mathbb{B}_n)\cap L^p(\mathbb{B}_n,dv)$$
where $dv$ is the normalized volume measure on $\mathbb{B}_n$.
Indeed, the weighted Bergman space $A^p_w(\mathbb{B}_n,dv)$ consists of holomorphic functions $f$ in the unit ball such that 
$$\Vert f\Vert^p_{A^p_w(\mathbb{B}_n,dv)}=\int_{\mathbb{B}_n}|f(z)|^pw(z)dv(z)<\infty.$$
Let $0<r<1$ and $z=(z_1,...,z_n)$. We mean by $rz=(rz_1,...,rz_n)$ and
$$\frac{z}{r}=\left(\frac{z_1}{r},...,\frac{z_n}{r}\right).$$
In this paper, we seek conditions on the weight function $w$ (defined on $\disk^n$ or $\mathbb{B}_n$) to guarantee that the dilations $f_r$ converge to $f$ in norm. Recall that if the weight function is radial, the problem is well-known; the difficulty arises when we consider non-radial weights. 
Since $f_r$ are holomorphic on a bigger domain, they can be approximated by polynomials so that $f$ can be approximated by polynomials too (provided that the weighted space contains the polynomials). We shall consider weights that satisfy the following Mergelyan-type condition: 
there is a non-negative
integer $k$ and an $r_0\in (0,1)$ such that
\begin{equation}\label{con-main}
r^kw\left (\frac{z}{r}\right )\le Cw(z),\quad\, r_0\le r<1,\, |z|<r.
\end{equation}
This condition on the weight function is universal; it does not depend on the underlying region.
In the next section we shall provide examples of non-radial weights satisfying this condition. 
We will prove some approximation theorems both for Bergman spaces and for Besov spaces. 
We should remark that
the definition of Besov spaces of holomorphic functions of several variables is slightly tricky; we shall present two different definitions for this spaces; the first definition that appears in Section \S 3 is taken from
\cite{zhu2}. The second definition, based on the notion of radial derivatives, is discussed in Section \S 4. 
To the best of our knowledge, the latter definition is new and were unable to trace it back in the literature.
Finally, in section 5, we introduce the concept of angular weights on polydisk. Theses are weights that just depend on the arguments of $z_1,...,z_n$. Similar approximation results are established for angular weights. To the best of the author's knowledge, the definition of angular weights are somewhat new, and have not yet been discussed in the literature.
\par
Polynomial approximation has many applications in operator theory of function spaces; see the papers \cite{chz}, \cite{kor}. Although approximation theory for non-radial weights has been discussed in earlier papers like \cite{hed1}, \cite{hed2}, but there are still many problems to be settled; see the comprehensive account on this topic in \cite{pelaez}. We close this section by mentioning several results in one variable case, see for instance \cite{abkar3}, \cite{abkar4}, \cite{abkar5}, \cite{abkar6}.

\section{The Bergman spaces}
It is a classical theorem that if the weight function on the unit disk is radial, that is $w(z)=w(|z|)$, then the polynomials are dense in the weighted Bergman space $A^2_w(\disk,dA)$. This was proved by S. Mergelyan \cite{mer} under the assumption that
$$\int_0^1rw(r)dr<\infty.$$
The same result is of course true for the Bergman spaces $A^p_w(\disk,dA)$, for $0<p<\infty$.
We assume now that $w$ is an arbitrary weight function (not necessarily radial) for which
$$\int_\disk w(z)dA(z)<\infty.$$
We shall see that under some mild condition on $w$, the polynomials are dense in the weighted Bergman spaces $A^p(\disk, d\mu_\alpha)$.
We begin with the following theorem.
\begin{theorem}\label{limsup}
Let $0<p<\infty$ and $d\mu_\alpha(z)=w(z)dV_\alpha(z)$ be a finite positive measure on the polydisk such that $w$ satisfies the condition (\ref{con-main}).
Then the polynomials are dense in the weighted Bergman space $A^p(\disk^n,d\mu_\alpha)$.
\end{theorem}
\begin{proof}
Let $f\in A^p(\disk^n, d\mu_\alpha)$, and consider the dilatations $f_r(z)=f(rz)$ for $0\le r<1$, and $z\in\disk^n$.
By a change of variables, we see that for non-negative $\alpha$ and big enough $r$ ($r\ge r_0$, see the condition (\ref{con-main}) above) we have
\begin{align*}\int_{\disk^n} |f_r(z)|^pd\mu_\alpha(z)&=
\frac{1}{r^{2n+k}}\int_{r\disk\times\cdots\times r\disk} |f(z)|^p r^k w\left(\frac{z}{r} \right)\prod_{k=1}^n\Big (\frac{r^2-|z_k|^2}{r^2}\Big )^\alpha dV(z)\\
&\le \frac{C}{r^{2n+k}}\int_{r\disk\times\cdots\times r\disk} |f(z)|^p w(z)dV(z)<\infty.
\end{align*}
We now use the dominated convergence theorem to conclude that
$$\limsup_{r\to 1^{-}} \int_{{\mathbb{D}}^n}|f_r|^pd\mu_\alpha(z)\le \int_{{\mathbb{D}}^n}|f|^pd\mu_\alpha(z).$$
This inequality together with Fatou's lemma implies that for each measurable subset $E\subseteq \disk^n$,
\begin{equation}\label{eq-lim-E}
\lim_{r\to 1^{-}} \int_{E}|f_r|^pd\mu_\alpha(z)= \int_{E}|f|^pd\mu_\alpha(z).
\end{equation}
By the continuity property of integral on its domain, there is a $\delta>0$ such that
$$\mu(E)<\delta \implies \int_{E}|f|^pd\mu_\alpha(z)<\epsilon.$$
It follows from Egorov's theorem that there is a subset $E\subset \disk^n$ such that $\mu_\alpha(\disk^n\setminus E)<\delta$ and $f_r\to f$ uniformly on $E$.
We now write
\begin{align}\label{2max}
\int_{{\disk}^n} |f_r-f|^p d\mu_\alpha &= \int_{E} |f_r-f|^p d\mu_\alpha +\int_{\disk\setminus E} |f_r-f|^p d\mu_\alpha  \nonumber \\
&\le \int_{E} |f_r-f|^p d\mua +2^p\int_{\pdisk\setminus E} (|f_r|^p+|f|^p) d\mua.
\end{align}
Due to the uniform convergence of $f_r$ to $f$ on $E$, the first integral on the right-hand side of (\ref{2max}) tends to zero as $r\to 1^-$.
As for the second integral, we use (\ref{eq-lim-E}) to obtain
\begin{align*}0\le \liminf_{r\to 1^{-}} \int_{\pdisk} |f_r-f|^pd\mua (z)&\le \limsup_{r\to 1^{-}} \int_{\pdisk} |f_r-f|^pd\mua (z)\\
&\le 2^{p+1}\int_{\pdisk\setminus E} |f|^p d\mua\\
&\le 2^{p+1}\epsilon,
\end{align*}
from which it follows that
\begin{equation}\label{ffrto0}
\lim_{r\to 1^{-}}\int_{\pdisk}|f_r-f|^pd\mua (z)=0.
\end{equation}
Given $\epsilon>0$, there is $0<r_0<1$ such that $\|f-f_{r_0}\|_{A^p(\disk^n,d\mua)}<\epsilon$.
Since $f_{r_0}$ is analytic in a bigger domain than the unit polydisk, we may use Mergelyan's theorem to approximate $f_{r_0}$ uniformly on $\disk^n$ by a polynomial $Q$. Since $\mua$ is a finite positive measure, we obtain
\begin{align*}\| f_{r_0}-Q\|^p_{A^p(\disk^n,d\mua)}&=\int_{\pdisk} |f_{r_0}-Q|^p d\mua \\
&\le \epsilon^p\int_{\pdisk} d\mua(z):=\epsilon^pC,
\end{align*}
from which it follows that
$$\| f-Q\|_{A^p(\disk^n,d\mua)}\le \epsilon(1+C^{1/p}).$$
This proves the theorem.
\end{proof}

\begin{remark}
The condition (\ref{con-main}) can be weakened in the following way: there is a non-negative function $g$ such that for sufficiently large $r$ we have $r^k w(z/r)\le g(z)$ and
$$\int_{\pdisk} |f(z)|^p g(z)dV_\alpha(z)<\infty.$$
Again we use the dominated convergence theorem to obtain
\begin{equation*}
\limsup_{r\to 1^{-}} \int_{E}|f_r|^pd\mua(z)\le \int_{E}|f|^pd\mua(z).
\end{equation*}
\end{remark}

\begin{theorem}\label{increasing}
Let $0<p<\infty$ and $d\mua(z)=w(z)dV_\alpha(z)$ be a finite positive measure on the polydisk such that for some non-negative integer $k$, the function
$$r\mapsto r^kw\left (\frac{z}{r}\right ),\quad |z|<r,$$
is increasing in $r$. Then the polynomials are dense in $A^p(\disk^n,wdV_\alpha)$.
\end{theorem}
\begin{proof}
It suffices to show that each function $f\in A^p(\disk^n,d\mua)$ satisfies
$$\limsup_{r\to 1^{-}} \int_{\pdisk}|f_r|^pd\mua(z)\le \int_{\pdisk}|f|^pd\mua(z).$$
To do this, we make a change of variables to get
\begin{equation*}\int_{\pdisk} |f_r(z)|^pd\mua(z)=
\frac{1}{r^{2n+k}}\int_{r\disk\times\cdots\times r\disk} |f(z)|^p r^k w\left(\frac{z}{r}\right )\prod_{j=1}^n\Big (\frac{r^2-|z_j|^2}{r^2}\Big )^\alpha dV(z).
\end{equation*}
Note that the functions
$$r\mapsto r^kw\left (\frac{z}{r}\right ),\quad r\mapsto \Big (\frac{r^2-|z_j|^2}{r^2}\Big )^\alpha$$
are increasing in $r$ (by the assumption and the fact that $\alpha$ is non-negative), so that
the monotone convergence theorem applies to yield
$$\limsup_{r\to 1^-}\int_{\pdisk} |f_r(z)|^pd\mua(z)=\int_{\pdisk} |f(z)|^pd\mua(z).$$
Now, we use an argument as in the proof of Theorem \ref{limsup} to get the result.
\end{proof}

\begin{example}
(a).~Let $\alpha>-1$ and $w(z)=\prod_{k=1}^n (\alpha+1)(1-|z_k|^2)^\alpha$.
$$r^kw\left (\frac{z}{r}\right )=\prod_{k=1}^n(\alpha+1)r^{k-2\alpha}(r^2-|z_k|^2)^\alpha\le w(z),$$
provided that $k\ge 2\alpha+1$. Now Theorem \ref{limsup} can be applied.\\
(b).~
It is easily seen that the Gaussian weight
$$w(z)=e^{-\beta |z|^2}=e^{-\beta (|z_1|^2+\cdots+|z_n|^2)},\quad \beta>0,$$
satisfies the condition of the above theorem for $k=0$, since for $0<r<r$,
\begin{align*}
w\left (\frac{z}{r}\right)&=e^{\frac{-\beta}{r^2}(|z_1|^2+\cdots+|z_n|^2)}\\
&\le e^{-\beta (|z_1|^2+\cdots+|z_n|^2)}\\ &=w(z).
\end{align*}
The same is true for the non-radial weight
$$w(z)=e^{-\beta (|x_1|^2+\cdots+|x_n|^2)},\quad \beta>0,$$
where $z=(z_1,...,z_n),\, z_k=x_k+iy_k,\, 1\le k\le n$.\\
Note that in some instances the function $r\mapsto w(z/r)$ is not increasing, but we may find some $k$ for which $r\mapsto r^kw(z/r)$
is increasing for $r$ bigger than $|z|$. For example, if we take $w(z)=\exp(|z|)$, then
$$\frac{d}{dr}w\left (\frac{z}{r}\right )=-\frac{|z|}{r^2}\exp\left (\frac{|z|}{r}\right )<0$$
while
$$\frac{d}{dr}\left[rw(\frac{z}{r})\right]=\left (1-\frac{|z|}{r}\right )\exp\left (\frac{|z|}{r}\right )>0,\quad |z|<r.$$
Note also that $r^2e^{|z|/r}$ is increasing for $r>1/2$ since
$$\frac{d}{dr}(r^2e^{|z|/r})=2re^{|z|/r}+r^2(\frac{-1}{r^2}e^{|z|/r})=(2r-1)e^{|z|/r}>0.$$
\end{example}

\section{The Besov spaces}
In the unit disk, the weighted Dirichlet space $\mathcal{D}^p_w,\, 1<p<\infty$, consists of analytic functions in the unit disk for which
$$\|f\|_{\mathcal{D}^p_w}=\left( |f(0)|^p+\int_{\disk}|f^\prime(z)|^pw(z)dA(z)\right )^{1/p}$$
is finite.
These spaces are particular cases of certain Banach spaces of analytic functions, namely, the weighted analytic Besov spaces. The weighted analytic Besov space $\mathcal{B}_w^p$ consists of analytic functions $f$ in the unit disk for which the integral
$$\int_{\disk}(1-|z|^2)^{p-2}|f^\prime(z)|^p w(z)dA(z)$$
is finite.
The norm of a function in the weighted Besov space is given by
\begin{equation*}
\| f\|_{\mathcal{B}^p_w}=\left (|f(0)|^p+\int_{\disk}(1-|z|^2)^{p-2}|f^\prime(z)|^p w(z) dA(z)\right )^{1/p}.
\end{equation*}
In higher dimensions, the definition of Besov spaces is more subtle. Let $dv(z)$ be the normalized volume measure in the unit ball $\mathbb{B}_n$ of $\mathbb{C}^n$, and let
$$d\tau(z)=\frac{dv(z)}{(1-|z|^2)^{n+1}}$$
be the M\"{o}bius invariant measure in $\mathbb{B}_n$. For $0<p<\infty$, the Besov space $\mathcal{B}^p$ consists of all holomorphic functions $f\in \mathbb{B}_n$ such that
$$(1-|z|^2)^N\frac{\partial^m f}{\partial z^m}(z)\in L^p(\mathbb{B}_n, d\tau),\quad |m|=N,$$
where $N$ is any fixed positive integer satisfying $pN>n$. According to \cite[Theorem 6.1]{zhu2}, the definition of $\mathcal{B}^p$ is independent of the choice of $N$. The "norm" is defined by
$$\Vert f\Vert_{\mathcal{B}^p}^p=\sum_{|m|\le N-1}\left |\frac{\partial^m f}{\partial z^m}(0)\right |^p+ \sum_{|m|=N}\int_{\mathbb{B}_n}
\left |(1-|z|^2)^N \frac{\partial^m f}{\partial z^m}(z)\right |^p d\tau (z).$$
When $p=2,\, \mathcal{B}^2$ plays the role of the Dirichlet space and is denoted by $\mathcal{D}^2$.

\begin{theorem}\label{limsup-dirichlet}
Let $d\mu(z)=w(z)d\tau(z)$ be a finite positive measure on the unit ball such that $w$ satisfies the condition (\ref{con-main}).
Then the polynomials are dense in $\mathcal{D}^2_w(\mathbb{B}_n, d\mu)$.
\end{theorem}
\begin{proof}
As for approximation, it is enough to work with the following expression for norm (the constant term does not play any role in approximation):
$$\|f\|^2_{\mathcal{D}^2_w}= \sum_{|m|=N}\int_{\mathbb{B}_n}
\left |(1-|z|^2)^N \frac{\partial^m f}{\partial z^m}(z)\right |^2 w(z)d\tau (z).$$
Note that
$$\|f_r\|^2_{\mathcal{D}^2_w}= \sum_{|m|=N}\int_{\mathbb{B}_n}
\left |(1-|z|^2)^N \frac{\partial^m f_r}{\partial z^m}(z)\right |^2 w(z)d\tau (z).$$
We now replace $z$ by $z/r$ to obtain
\begin{align*}\|f_r\|^2_{\mathcal{D}^2_w}&= \sum_{|m|=N}\frac{1}{r^k}\int_{r\mathbb{B}_n}
\left |\Big (\frac{r^2-|z|^2}{r^2}\Big )^N r^{|m|} \frac{\partial^m f}{\partial z^m}(z)\right |^2 r^k w\left (\frac{z}{r}\right )\frac{r^2}{(r^2-|z|^2)^{n+1}} dv(z)\\
&\le C\frac{r^{2+2|m|}}{r^{k+4N}}\sum_{|m|=N}\int_{r\mathbb{B}_n}
\left |\big (1-|z|^2\big )^N  \frac{\partial^m f}{\partial z^m}(z)\right |^2 w(z) \frac{dv(z)}{(r^2-|z|^2)^{n+1}} \\
&=C\frac{r^{2+2|m|}}{r^{k+4N}}\sum_{|m|=N}\int_{\mathbb{B}_n}\chi(r\mathbb{B}_n)(z)
\left |\big (1-|z|^2\big )^N\frac{\partial^m f}{\partial z^m}(z)\right |^2 w(z) \frac{dv(z)}{(r^2-|z|^2)^{n+1}}\\
&\le C\frac{r^{2+2|m|}}{r^{k+4N}}\sum_{|m|=N}\int_{\mathbb{B}_n}
\left |\big (1-|z|^2\big )^N\frac{\partial^m f}{\partial z^m}(z)\right |^2 w(z)\frac{dv(z)}{(r^2-|z|^2)^{n+1}}.
\end{align*}
Since the function
$$r\mapsto \frac{1}{(r^2-|z|^2)^{n+1}},\quad |z|<r,$$
is monotone decreasing in $r$,
the last integral converges ar $r\to 1^-$, so that a version of the dominated convergence theorem can be applied to the first integral:
\begin{align*}\limsup_{r\to 1^-}\|f_r\|^2_{\mathcal{D}^2_w}&= \limsup_{r\to 1^-}\sum_{|m|=N}\frac{1}{r^k}\int_{r\mathbb{B}_n}
\left |\Big (\frac{r^2-|z|^2}{r^2}\Big )^N r^{|m|} \frac{\partial^m f}{\partial z^m}(z)\right |^2 r^k w\left (\frac{z}{r}\right )\frac{r^2}{(r^2-|z|^2)^{n+1}} dv(z)\\
&=\sum_{|m|=N}\int_{\mathbb{B}_n}
\left |(1-|z|^2)^N \frac{\partial^m f}{\partial z^m}(z)\right |^2 w(z)\frac{dv(z)}{(1-|z|^2)^{n+1}}\\
&=\|f\|^2_{\mathcal{D}^2_w}.
\end{align*}
The rest of the proof goes in the same way as in the proof of Theorem \ref{limsup}; we just replace $f$ and $f_r$, by $\frac{\partial^m f}{\partial z^m}$ and $\frac{\partial^m f_r}{\partial z^m}$, respectively, to obtain
$$\lim_{r\to 1^-}\|f_r -f\|^2_{\mathcal{D}^2_w}=\lim_{r\to 1^-}
\sum_{|m|=N}\int_{\mathbb{B}_n}
\left |(1-|z|^2)^N \left (\frac{\partial^m f}{\partial z^m}(z)- \frac{\partial^m f_r}{\partial z^m}(z)\right )\right |^2 w(z)d\tau (z)=0.$$
On the other hand, each dilation $f_r$ can be approximated by polynomials (see \cite[Theorem 6.2]{zhu2}), therefore, $f$ can be approximated by polynomials.
\end{proof}

We now state a similar statement for the analytic Besov spaces.
\begin{theorem}\label{limsup-dirichlet}
Let $d\mu(z)=w(z)d\tau(z)$ be a finite positive measure on the unit ball such that $w$ satisfies the condition (\ref{con-main}).
Then the polynomials are dense in $\mathcal{B}^p_w(\mathbb{B}_n, d\mu)$.
\end{theorem}
\begin{proof}
We begin by inspecting the expression
$$\|f_r\|^p_{\mathcal{B}^p_w}=
 \sum_{|m|=N}\int_{\mathbb{B}_n}
\left |(1-|z|^2)^N \frac{\partial^m f_r}{\partial z^m}(z)\right |^p w(z)d\tau (z).$$
Similar to the preceding case, we have
\begin{align*}\|f_r\|^p_{\mathcal{B}^p_w}&= \sum_{|m|=N}\frac{1}{r^k}\int_{r\mathbb{B}_n}
\left |\Big (\frac{r^2-|z|^2}{r^2}\Big )^N r^{|m|} \frac{\partial^m f}{\partial z^m}(z)\right |^p r^k w\left (\frac{z}{r}\right )\frac{r^2}{(r^2-|z|^2)^{n+1}} dv(z)\\
&\le C\frac{r^{2+p|m|}}{r^{k+2pN}}\sum_{|m|=N}\int_{\mathbb{B}_n}
\left |\big (1-|z|^2\big )^N\frac{\partial^m f}{\partial z^m}(z)\right |^p w(z)\frac{dv(z)}{(r^2-|z|^2)^{n+1}}.
\end{align*}
Again, this argument shows that
$$\limsup_{r\to 1^-}\|f_r\|^p_{\mathcal{B}^p_w}\le \|f\|^p_{\mathcal{B}^p_w},$$
from which the result follows.
\end{proof}

\section{Besov spaces: alternative definition}
Another approach in the study of Besov spaces in several complex variables is to use the concept of radial derivative.
Let $f$ be a holomorphic function in the unit ball $\mathbb{B}_n$. The radial derivative of $f$ is defined by
$$Rf(z)=\sum_{k=1}^n z_k\frac{\partial f}{\partial z_k}(z).$$
According to \cite{zhu2}, if
\begin{equation}\label{homg-exp}
f(z)=\sum_{k=0}^\infty f_k(z)
\end{equation}
is the homogenous expansion of $f$ (each $f_k$ is a polynomial of degree $k$ in $z_1$,...,$z_n$), then
$$Rf(z)=\sum_{k=1}^\infty kf_k(z).$$
The weighted analytic Besov space $\mathbf{B}_w^p(\disk^n)$ consists of analytic functions $f$ on a neighborhood of polydisk for which the integral
$$\int_{\disk^n}\Big(\prod_{j=1}^n(1-|z_j|^2)^{p-2}\Big)|Rf(z)|^p w(z)dV(z)$$
is finite.
The norm of a function in $\mathbf{B}_w^p(\disk^n)$ is given by
\begin{equation*}
\| f\|^p_{\mathbf{B}^p_w(\disk^n)}=|f(0)|^p+\int_{\disk^n}\Big(\prod_{j=1}^n(1-|z_j|^2)^{p-2}\Big)|Rf(z)|^p w(z)dV(z).
\end{equation*}

\begin{theorem}\label{besov-polydisk-radial}
Let $2\le p<\infty$ and $d\mu(z)=w(z)dV(z)$ be a finite positive measure on the polydisk such that $w$ satisfies the condition (\ref{con-main}).
Then the polynomials are dense in the weighted analytic Besov space $\mathbf{B}_w^p(\disk^n)$.
\end{theorem}
\begin{proof}
As we pointed out earlier, the constant term in the definition of norm does not play any role in approximation, so that we just look at
 \begin{equation*}
\| f_r\|^p_{\mathbf{B}^p_w(\disk^n)}=\int_{\disk^n}\Big(\prod_{j=1}^n(1-|z_j|^2)^{p-2}\Big)|Rf_r(z)|^p w(z)dV(z).
\end{equation*}
We first note that by (\ref{homg-exp})
$$f_r(z)=f(rz)=\sum_{m=0}^\infty f_m(rz),$$
from which it follows that
$$(Rf_r)(z)=\sum_{m=0}^\infty mf_m(rz).$$
Replacing $z$ by $z/r$ we obtain
\begin{align*}
\| f_r\|^p_{\mathbf{B}^p_w(\disk^n)}&=\int_{\disk^n}\Big[\prod_{j=1}^n(1-|z_j|^2)^{p-2}\Big]|Rf_r(z)|^p w(z)dV(z)\\
&=\frac{1}{r^{k+2n}}\int_{r\disk\times \cdots \times r\disk}\Big[\prod_{j=1}^n\Big (\frac{r^2-|z_j|^2}{r^2}\Big )^{p-2}\Big]|Rf(z)|^p r^k w(\frac{z}{r})dV(z)\\
&\le \frac{C}{r^{k+2n}}\int_{r\disk\times \cdots \times r\disk}\Big[\prod_{j=1}^n\Big (\frac{r^2-|z_j|^2}{r^2}\Big)^{p-2}\Big]|Rf(z)|^p w(z)dV(z).
\end{align*}
Since $p-2\ge 0$, each of the functions
$$r\mapsto \Big(\frac{r^2-|z_j|^2}{r^2}\Big )^{p-2}$$
is increasing in $r$, so that we can apply the generalized version of dominated convergence theorem to get
$$\lim_{r\to 1^{-1}}\| f_r\|^p_{\mathbf{B}^p_w(\disk^n)}=\| f\|^p_{\mathbf{B}^p_w(\disk^n)}.$$
This latter implies that the dilations $f_r$ tend to $f$ in norm;
$$\| f-f_r\|^p_{\mathbf{B}^p_w(\disk^n)}=0,$$
which leads to the desired result.
\end{proof}

\section{Angular weights on polydisk}
In contrast to radial weights defined on the unit disk, we may consider weights that depend only on the argument of $z$; that is,
$w(re^{i\theta})=w(\theta)$. We may call such weights {\it angular weights}. For example, given $\alpha>0$,
$$w(z)=w(re^{i\theta})=(4\pi^2-\theta^2)^\alpha,\quad 0\le \theta <2\pi,$$
is an angular weight on the unit disk. Similarly, a positive continuous function $w$ on the polydisk $\disk^n$ is said to be angular if
$$w(r_1e^{i\theta_1},...,r_ne^{i\theta_n})=w(\theta_1,...,\theta_n),$$
meaning that $w$ just depends on the arguments of the components of $z\in\disk^n$.\\

\noindent{\bf Some notations}\\
For simplicity, given an $n$-tuple $\mathbf{r}=(r_1,...,r_n)$, we write $d\mathbf{r}=dr_1\cdots dr_n$, and for a given weight function $w$, we write $w(\mathbf{r})$ instead of $\omega(r_1,...,r_n)$. In the same manner, $w(\mathbf{\theta})$ is a short form for $w(\theta_1,...,\theta_n)$, and
$d\mathbf{\theta}$ stands for $d\theta_1\cdots d\theta_n$. The distinguished boundary of $\disk^n$ is denoted by $\mathbb{T}^n$, and $|m|=m_1+\cdots+m_n$ where $m=(m_1,...,m_n)$ is a multi-index consisting of non-negative integers.

\begin{proposition}\label{angular}
Let $w$ be an angular weight function on the polydisk that satisfies
$$\int_{\mathbb{T}^n}w(\mathbf{\theta})d\mathbf{\theta} <\infty.$$
Then the Taylor polynomials of each function in $A^2_w(\disk^n,dV_\alpha)$ tend to the same function in norm. In particular,
the polynomials are dense in $A^2_w(\disk^n,dV_\alpha)$.
\end{proposition}
\begin{proof}
Let $m=(m_1,...,m_n)$ be a multi-index, and
$$f(z)=\sum_{|m|\ge 0} a_mz^m=\sum_{|m|\ge 0} a_{m_1,...,m_n}z_1^{m_1}\cdots z_n^{m_n}$$
be a function in $A^2_w(\disk^n,dV_\alpha)$. Let $I^n$ stand for the $n$ times Cartesian product of the unit interval $[0,1]$.
Then using polar coordinates we have
\begin{align*}\|f\|^2_{A^2_w(\disk^n,dV_\alpha)}&=\int_{I^n}\int_{\mathbb{T}^n}\sum_{|m|\ge 0}|a_m|^2\prod_{k=1}^{n}\left(r_k^{2m_k}(1-r_k^2)^\alpha r_kdr_k\right )\, w(\mathbf{\theta}) d\mathbf{\theta}\\
&=\Big(\int_{\mathbb{T}^n}w(\mathbf{\theta}) d\mathbf{\theta}\Big )\sum_{|m|\ge 0}\frac{|a_m|^2\,\Gamma(m_1+1)\cdots\Gamma(m_n+1)(\Gamma(\alpha+1))^n}{2^n\Gamma(m_1+\alpha+1)\cdots\Gamma(m_n+\alpha+1)}.
\end{align*}
This implies that for $p_k(z)=\sum_{0\le |m|\le k}a_mz^m$ we have
$$\|f-p_k \|^2_{A^2_w(\disk^n,dV_\alpha)}=\Big(\int_{\mathbb{T}^n}w(\mathbf{\theta}) d\mathbf{\theta}\Big )\sum_{|m|\ge k+1}
\frac{|a_m|^2\,\Gamma(m_1+1)\cdots\Gamma(m_n+1)(\Gamma(\alpha+1))^n}{2^n\Gamma(m_1+\alpha+1)\cdots\Gamma(m_n+\alpha+1)}.$$
Now, we let $k$ tend to infinity, so that $p_k\to f$ in $A^2_w(\disk^n,dV_\alpha)$.
\end{proof}

This result can be generalized in the following way.
\begin{theorem}\label{angular-main}
Let $w$ be an angular weight function on the polydisk that satisfies
$$\int_{\mathbb{T}^n}w(\mathbf{\theta})d\mathbf{\theta} <\infty.$$
Then the polynomials are dense in $A^p_w(\disk^n,dV_\alpha),\, 0<p<\infty$.
\end{theorem}
\begin{proof}
Since $w$ is angular, we see that $w(z)=w(\frac{z}{r})$, so that
\begin{align*}\|f_r\|^p_{A^p_w(\disk^n,dV_\alpha)}&=\int_{\disk^n}|f_r(z)|^pw(z)\prod_{k=1}^n(1-|z_k|^2)^\alpha dV(z)\\
&= \frac{1}{r^{2n}}\int_{r\disk^n}|f(z)|^p w(\frac{z}{r})\prod_{k=1}^n\left(\frac{r^2-|z_k|^2}{r^2}\right)^\alpha dV(z)\\
&=\frac{1}{r^{2n}}\int_{r\disk^n}|f(z)|^p w(z)\prod_{k=1}^n\left(\frac{r^2-|z_k|^2}{r^2}\right)^\alpha dV(z).
\end{align*}
Again, we use $\alpha\ge 0$ and the fact that the function
$$r\mapsto \Big (\frac{r^2-|z_k|^2}{r^2}\Big )^\alpha,\quad |z_k|\le r,$$
is increasing in $r$ to obtain
$$\limsup_{r\to 1^-}\int_{\disk^n}|f_r(z)|^pw(z)dV_\alpha (z)=\int_{\disk^n}|f(z)|^pw(z)dV_\alpha(z),$$
As in the proof of Theorem \ref{limsup}, this equality leads to
$$\|f_r-f\|_{A^p_w(\disk^n,dV_\alpha)}=0,\quad r\to  1^-,$$
from which the result follows.
\end{proof}

We now consider weights that are products of a radial and an angular weight.
\begin{proposition}\label{radial-angular}
Let $w(r_1e^{i\theta_1},...,r_ne^{i\theta_n})=\omega(r_1,...,r_n)\nu(\theta_1,...,\theta_n),$
where $\omega$ and $\nu$ are two weights on $\disk^n$ satisfying
$$\int_{I^n} \Big[\prod_{k=1}^{n}(1-r_k^2)^\alpha r_k\Big ]\omega(\mathbf{r})d\mathbf{r}<\infty,\quad \int_{\mathbb{T}^n}\nu (\mathbf{\theta})d\mathbf{\theta}<\infty.$$
Then the Taylor polynomials of each function in $A^2_w(\disk^n,dV_\alpha)$ tend to the same function in norm.
In particular, the polynomials are dense
in $A^2_w(\disk^n,dV_\alpha)$.
\end{proposition}
\begin{proof}
A computation shows that
\begin{align*}\|f\|^2_{A^2_w(\disk^n,dV_\alpha)}&=\int_{I^n}\int_{\mathbb{T}^n}\sum_{|m|\ge 0}|a_m|^2\prod_{k=1}^{n}\left(r_k^{2m_k}(1-r_k^2)^\alpha r_kdr_k\right )\omega(\mathbf{r})\nu (\mathbf{\theta})d\mathbf{r}d\mathbf{\theta}\\
&=\Big(\int_{\mathbb{T}^n}\nu (\mathbf{\theta})d\mathbf{\theta}\Big)\sum_{|m|\ge 0}|a_m|^2 \int_{I^n} \Big[\prod_{k=1}^{n}r_k^{2m_k}(1-r_k^2)^\alpha r_k\Big ] \omega(\mathbf{r})d\mathbf{r}\\
&=\Big(\int_{\mathbb{T}^n}\nu (\mathbf{\theta})d\mathbf{\theta}\Big)\sum_{|m|\ge 0}|a_m|^2 \gamma_m,
\end{align*}
where
\begin{align*}\gamma_m &=\int_{I^n} \Big[\prod_{k=1}^{n}r_k^{2m_k}(1-r_k^2)^\alpha r_k\Big ]\omega(\mathbf{r})d\mathbf{r}\\
&\le \int_{I^n} \Big[\prod_{k=1}^{n}(1-r_k^2)^\alpha r_k\Big ]\omega(\mathbf{r})d\mathbf{r}<\infty.
\end{align*}
Now, for $p_k(z)=\sum_{0\le |m|\le k}a_mz^m$ we have
$$\|f-p_k \|^2_{A^2_w(\disk^n,dV_\alpha)}=\Big(\int_{\mathbb{T}^n}\nu (\mathbf{\theta})d\mathbf{\theta}\Big)\sum_{|m|\ge k+1}^{\infty}|a_m|^2\gamma_m,$$
which tends to zero as $k\to\infty$.
\end{proof}
\begin{theorem}\label{radial-angular-main}
Let $w(r_1e^{i\theta_1},...,r_ne^{i\theta_n})=\omega(r_1,...,r_n)\nu(\theta_1,...,\theta_n),$
where $\omega$ and $\nu$ are two weights on $\disk^n$ satisfying
$$\int_{I^n} \Big[\prod_{k=1}^{n}(1-r_k^2)^\alpha r_k\Big ]\omega(\mathbf{r})d\mathbf{r}<\infty,\quad \int_{\mathbb{T}^n}\nu (\mathbf{\theta})d\mathbf{\theta}<\infty,$$
and $s^k\omega(r/s)\le C\omega(r)$ for some integer $k\ge 0$.
Then the polynomials are dense in the weighted Bergman space $A^p_w(\disk^n,dV_\alpha),\, 0<p<\infty$.
\end{theorem}
\begin{proof}
Again, we see that for $z=(r_1e^{i\theta_1},...,r_ne^{i\theta_n})$,
\begin{align*}\|f_s\|^p_{A^p_w(\disk^n,dV_\alpha)}&=\int_{\disk^n}|f_s(z)|^p\omega(r)\nu(\theta)\prod_{j=1}^n(1-|z_j|^2)^\alpha dv(z)\\
&= \frac{1}{s^{2n+k}}\int_{s\disk^n}|f(z)|^p s^k\omega(r/s)\prod_{j=1}^n\left(\frac{s^2-|z_j|^2}{s^2}\right)^\alpha\nu(\theta)dv(z)\\
&\le \frac{C}{s^{2n+k}}\int_{s\disk^n}|f(z)|^p \omega(r)\prod_{j=1}^n\left(\frac{s^2-|z_j|^2}{s^2}\right)^\alpha\nu(\theta)dv(z).\\
\end{align*}
Using the fact that for $\alpha\ge 0$ and $j$ fixed,
$$s\mapsto \Big (\frac{s^2-|z_j|^2}{s^2}\Big )^\alpha,\quad |z_j|\le s,$$
is increasing in $s$ we obtain
$$\limsup_{s\to 1^-}\|f_s\|^p_{A^p_w(\disk^n,dV_\alpha)}\le \|f\|^p_{A^p_w(\disk^n,dV_\alpha)}.$$
We now use the method introduced in the proof of Theorem \ref{limsup} to get
$$\|f_s-f\|_{A^p_w(\disk^n,dV_\alpha)}=0,\quad s\to 1^-,$$
which in turn leads to the desired result.
\end{proof}







\section{Declarations}

\noindent{\bf Ethical approval}\\
Not allpicable\\

\noindent{\bf Competing interests}\\
The author declares no competing interests.\\

\noindent{\bf Authors contribution}\\
Not allpicable\\

\noindent{\bf Funding}\\
Not allpicable\\

\noindent{\bf Availability of data and materials} \\
Data sharing is not applicable to this article as no data sets were generated or analyzed during the current study.

\end{document}